\pdfoutput=1
\RequirePackage{ifpdf}
\ifpdf 
\documentclass[pdftex]{sigma}
\else
\documentclass{sigma}
\fi

\numberwithin{equation}{section}

\newcommand\pFq[5]{{}_{#1}F_{#2} \left[\begin{matrix} #3 \\ #4 \end{matrix}; #5\right]}
\newcommand\pFqS[5]{{}_{#1}F_{#2} \left[\begin{smallmatrix} #3 \\ #4 \end{smallmatrix}; #5\right]}

\newtheorem{Theorem}{Theorem}[section]
\newtheorem{Corollary}[Theorem]{Corollary}
\newtheorem{Lemma}[Theorem]{Lemma}
\newtheorem{Proposition}[Theorem]{Proposition}
\newtheorem{Conjecture}[Theorem]{Conjecture}
 { \theoremstyle{definition}

\newtheorem{Example}[Theorem]{Example}
\newtheorem{Remark}[Theorem]{Remark} }

\begin{document}

\allowdisplaybreaks

\newcommand{\arXivNumber}{1609.06157}

\renewcommand{\thefootnote}{}

\renewcommand{\PaperNumber}{063}

\FirstPageHeading

\ShortArticleName{$d$-Orthogonal Analogs of Classical Orthogonal Polynomials}

\ArticleName{$\boldsymbol{d}$-Orthogonal Analogs of Classical Orthogonal\\ Polynomials\footnote{This paper is a~contribution to the Special Issue on Orthogonal Polynomials, Special Functions and Applications (OPSFA14). The full collection is available at \href{https://www.emis.de/journals/SIGMA/OPSFA2017.html}{https://www.emis.de/journals/SIGMA/OPSFA2017.html}}}

\Author{Emil HOROZOV~$^{\dag\ddag}$}

\AuthorNameForHeading{E.I.~Horozov}

\Address{$^\dag$~Department of Mathematics and Informatics, Sofia University,\\
\hphantom{$^\dag$}~5 J.~Bourchier Blvd., Sofia 1126, Bulgaria}
\EmailD{\href{mailto:horozov@fmi.uni-sofia.bg}{horozov@fmi.uni-sofia.bg}}

\Address{$^\ddag$~Institute of Mathematics and Informatics, Bulg. Acad. of Sci.,\\
\hphantom{$^\ddag$}~Acad. G.~Bonchev Str., Block 8, 1113 Sofia, Bulgaria}

\ArticleDates{Received October 01, 2017, in final form June 13, 2018; Published online June 26, 2018}

\Abstract{Classical orthogonal polynomial systems of Jacobi, Hermite and Laguerre have the property that the polynomials of each system are eigenfunctions of a second order ordinary differential operator. According to a famous theorem by Bochner they are the only systems on the real line with this property. Similar results hold for the discrete orthogonal polynomials. In a recent paper we introduced a natural class of polynomial systems whose members are the eigenfunctions of	 a differential operator of higher order and which are orthogonal with respect to $d$ measures, rather than one. These polynomial systems, enjoy a~number of properties which make them a natural analog of the classical orthogonal polynomials. In the present paper we continue their study. The most important new properties are their hypergeometric representations which allow us to derive their generating functions and in some cases also Mehler--Heine type formulas.}
	
\Keywords{$d$-orthogonal polynomials; finite recurrence relations; bispectral problem; gene\-ralized hypergeometric functions; generating functions}
	
\Classification{34L20; 30C15; 33E05}

\renewcommand{\thefootnote}{\arabic{footnote}}
\setcounter{footnote}{0}

\section{Introduction} \label{intro}

This paper is a natural continuation of the study initiated in \cite{Ho3} on the basis of several classes of examples in \cite{Ho1, Ho2}. There we constructed large families of polynomial systems that were called {\it $d$-orthogonal polynomials with the Bochner property.}

The terminology ``Bochner's property'' derives from the Bochner's theorem~\cite{Bo} mentioned in the abstract and means that the polynomials are eigenfunction of a differential operator. We recall that, by definition, general $d$-orthogonal polynomials are polynomial systems $P_n(x)$, $ n =0,1, 2, \ldots$, $\deg (P_n) =n$ iff there exist $d$ linear functionals $\mathcal{L}_j$, $j =0,\ldots,d-1$ on the space of all polynomials~${\mathbb C}[x]$ such that
\begin{gather*}
 \begin{cases}
 \mathcal{L}_j (P_nP_m) = 0, & m > nd+ j, \ n \geq 0, \\
 \mathcal{L}_j(P_nP_{nd+ j}) \neq 0, & n \geq 0,
 \end{cases}
 \end{gather*}
 for each $j \in N_{d} := \{0, \ldots, d-1 \}$. When $d = 1$ this is the ordinary notion of orthogonal polynomials. The orthogonality is connected with $d$ functionals rather than with only one. According to~\cite{Ma, VIs} the above property is equivalent to the existence of a~linear recurrence relation of the form
\begin{gather*}
xP_n(x) = P_{n+1} + \sum_{j=0}^{d} \gamma_j(n) P_{n-j}(x).
\end{gather*}
Here and later we use mostly monic polynomials, i.e., whose coefficient at the highest degree is~1.

The $d$-orthogonal polynomials and the more general class of the so-called multiple orthogonal polynomials have been intensively studied in the last 30 years due to their intriguing properties and applications, cf., e.g., \cite{Apt, ApKu, BDK} and~\cite[Chapter~23]{Is} and further references in the cited literature. In particular they have applications to random matrices \cite{Ku, KZh}, simultaneous Pad\'e approximations~\cite{DBr}, number theory \cite{ BR,Beu, So} (which in fact go back to Hermite), etc.

Notice that the classical orthogonal polynomials have a number of properties that are missing \textit{in general} for the rest of the polynomial systems. Here we list some of them:
\begin{itemize}\itemsep=0pt
\item they are eigenfunctions of an ordinary differential operator,
\item they have \textit{explicit differential} ladder operators (operators raising or lowering the index),
\item they can be presented in terms of hypergeometric functions,
\item they can be presented via Rodrigues formulas,
\item there are Pearson's equations for the weights of their measures,
\item they possess the Hahn's property, i.e., the polynomial system of their derivatives are again orthogonal polynomials.
\end{itemize}

The class of polynomial systems that we introduced in \cite{Ho3} also shares these properties. Some of them were established in the cited paper, e.g., \textit{explicit differential} ladder operators\footnote{As one of the referees kindly pointed to me in fact each polynomial system has ladder operators, cf.~\cite{BC}. However in~\cite{Ho3} we have given explicit differential ones. } and Rodrigues-like formulas, apart of the differential equation. All these properties were direct consequences of our main construction. Other properties will be found here~-- their hypergeometric representations, generating functions and in some cases~-- Mehler--Heine formulas.

In another project we intend to obtain also the weights, defining the functionals $\mathcal{L}_j$, together with the Pearson equations for them and show their connection to biorthogonal ensembles, cf.~\cite{Bor, KZh}. All these properties make them close analogs of the classical orthogonal polynomials.

In fact there are other polynomial systems that are analogs of classical orthogonal polyno\-mials. In \cite{ABV, VAC} the authors take another direction to generalize classical orthogonal polynomials. Namely they use the weights of the latter to produce a vector of weights. Their polynomial systems were later used in the study of different random matrix models~-- non-intersecting Brownian motion, matrix models with an external
source, two-matrix models, cf.~\cite{ABK, BDK, Ku}.

 The main tools we used in \cite{Ho3} are the automorphisms of non-commutative algebras. We explain the construction for the case of the first Weyl algebra~$W_1$. It can be realized as the algebra of differential operators with polynomial coefficients in one variable. $W_1$ acts on the space of polynomials ${\mathbb C}[x]$. We consider the simplest polynomial system $\{\psi_n(x) = x^n, \, n=0, 1, \ldots\}$, and the differential operator $H= x\partial$ which satisfies
 \begin{gather*}
 H\psi_n(x) =n\psi_n(x), \qquad \partial \psi_n(x) = n\psi_{n-1}(x) \qquad \text{and} \qquad x\psi_n(x) = \psi_{n+1}(x).
 \end{gather*}
Take any polynomial $q(\partial)$ in $\partial$, where $\partial := \frac{{\rm d}}{{\rm d}x}$. Below for simplicity we take $q(\partial) = -\partial^l/l$, $l\in {\mathbb N}$. It defines an automorphism $\sigma = e^{\operatorname{ad}_{q(\partial)}}$ of $W_1$, acting on $A \in W_1$ as
 \begin{gather*}
 \sigma(A) = e^{-\operatorname{ad}_{\partial^l/l}} A = \sum_{j=0}^{\infty} \frac{\operatorname{ad}^j_{-\partial^l/l}(A)}{j!},
 \end{gather*}
 where $\operatorname{ad}_A(B) = [A,B]$. It is easy to see that the sum is finite. The images of the above operators are
\begin{gather*}
 \sigma (H) = H - \partial^l, \qquad \sigma (\partial) = \partial, \qquad \sigma (x) = x -\partial^{l-1}.
\end{gather*}
 If we define the polynomials
 \begin{gather*}
 P_n(x) = e^{-\partial^l/l} \psi_n(x) = \sum_{j=0}^{\infty}\frac{(-\partial^l/l)^jx^n}{j!},
 \end{gather*}
and put $L= \sigma (H)$ we easily see that
 \begin{gather*}
 LP_n(x) = n P_n(x), \\
xP_n(x) = P_{n+1}(x) + n(n-1)\cdots (n-l+2)P_{n-l}(x).
 \end{gather*}
{\it The main point is that we constructed simultaneously the polynomial system $\{P_n(x)\}$, the differential operator $L$ and the finite-term recurrence.} Notice that for $l=2$ these are the Hermite polynomials. For arbitrary $l$ these are the Gould--Hopper polynomials~\cite{GH}. See also the examples.

The same procedure can be repeated with other algebras. We can take instead of $\partial$ an operator of the form $G= R(H)\partial$, where $R(H)$ is any polynomial in $H$. Then instead of $W_1$ we can take the algebra spanned by $G$, $H$, $x$.

The case of discrete orthogonal polynomials can be treated exactly in the same manner, realizing $W_1$ by difference operators, and starting with $\psi(x, n) = x(x-1)\cdots(x-n+1)$, cf.\ the next section or~\cite{Ho3}.

In the present paper we study further the $d$-orthogonal polynomials with the Bochner property adding to the tools some new arguments which are not present in~\cite{Ho3}. Let us first list the new properties which we discuss below.

The most fundamental one, of which the rest are consequences, is \textit{the hypergeometric representation} of the class of $d$-orthogonal polynomials with the Bochner property corresponding to the special case $q(G) = \rho G^l$, $l \in {\mathbb N}$, $\rho \in {\mathbb C}$, which we obtain both in the continuous and the discrete cases. Moreover, in both cases we provide two different representations in terms of hypergeometric functions. One of these representations has the advantage that the corresponding formula has the same values of hypergeometric parameters for all positive integers~$n$. However, the other formula, in which the corresponding parameters depend on $n$ $({\rm mod}~l)$ seems to be more useful in applications.

 The hypergeometric representations show that the class of $d$-orthogonal polynomials under consideration is very similar to the Gould--Hopper polynomials \cite{GH}, which correspond to $G= \partial_x$ and can also be included as a special case of our construction. For this reason, we call the $d$-orthogonal polynomials introduced above the {\it generalized Gould--Hopper $($GGH$)$ polynomials}.

 One immediate consequence of the above hypergeometric representations is that these special polynomial systems naturally split in~$l$ families, each originating from the initial differential (difference) operator~$G$, exactly as in the case of the Hermite polynomials, which also is a~particular case of our general construction. Recall that the latter are naturally subdivided into two families: the even-indexed $H_{2n}(x)$, $n =0, 1, 2 \ldots$ and the odd-indexed ones $H_{2n +1}(x)$, $n =0, 1, 2 \ldots$. The corresponding representations are $H_{2n}(x)= L_n^{(-1/2)}\big(x^2\big) $ and $H_{2n}(x)= xL_n^{(1/2)}\big(x^2\big) $ (up to multiplicative constants), where $L^{(\alpha)}_n$ are the generalized Laguerre polynomials.

It is worth to point out that the above mentioned hypergeometric representations also use our main construction of the $d$-orthogonal polynomials in \cite{Ho3}, but they require some additional transformations. The point is that in these representations of the polynomials $P_n(x)$ each summand naturally corresponds to a summand in a generalized hypergeometric series. In the differential case when one uses $q(G) = G$, the hypergeometric representations were known before, see~\cite{BCD2}. Also the case of the Gould--Hopper polynomials is known \cite{LCS}. However, for $q(G) = \rho G^l$, $l >1,$ our formulas are new. In the discrete case, the formulas for the families corresponding to $G = a\Delta$, $a \in {\mathbb C}$ and $q(G) = G^l$, $l >1$ can be found in \cite{BCZ}. All other formulas, except for the Charlier and Meixner polynomials are new.

Our next goal is to find the \textit{generating functions} for the GGH polynomials. We present two such formulas, both based on the second hypergeometric representation. For the first family of generating functions, we use the well-known method to obtain the generating functions for the classical orthogonal polynomials. For the second family, we apply a formula due to Srivastava~\cite{Sr1}. Some of the $d$-orthogonal polynomials are known to have a generating function. These are exactly the ones mentioned above and which can be found in~\cite{BCD2}.

Finally we find Mehler--Heine type asymptotic formulas \cite{Hei, Meh}, which are again based on the second hypergeometric representation. In case of the multiple orthogonal polynomials, such formulas can be found in \cite{Tak, VAss}. In particular~\cite{VAss} contains our first result. We do not treat the discrete case, where as kindly pointed to me one of the referees the corresponding notion is local separate convergence, although some authors use again Mehler--Heine type asymptotics, cf.~\cite{Dom}.

It is essential to mention that the constructions of all systems of $d$-orthogonal polynomials in \cite{Ho3} are based on the so-called \textit{bispectral problem}, see~\cite{BHY1, BHY2, DG}. Namely, these polynomials are the eigenfunctions for two linear operators~-- the first one is differential (difference)~$L(x)$ in the variable~$x$, and the second one is the operator $\Lambda(n)$ in the variable~$n$, corresponding to the finite recurence relation:
\begin{gather*}
L(x) P_n(x) = nP_n(x), \qquad xP_n(x) = \Lambda(n)P_n(x).
\end{gather*}
It turned out that both the generating functions and the Mehler--Heine asymptotics in the continuous case are expressed in terms of hypergeometric functions of the form
\begin{gather*}
 \pFq{0}{q}{-}{\alpha_1, \alpha_2, \ldots, \alpha_q}{xt}.
 \end{gather*}
The latter functions in the form of Meijer's $G$-functions appear in an entirely different bispectral problem~-- for which both variables $x$ and $n$ are continuous, see~\cite{BHY2}, where they are called generalized Bessel functions. It seems that this is not a mere coincidence but could be exploited further. In particular, the already known results for the Bessel bispectral functions and their Darboux transformations could be used as a model in the study of our $d$-orthogonal polynomials with the Bochner property.

Apart from Darboux transformations, some other possible directions of continuations of the present studies might include the so-called linearization problem (Clebsch--Gordan coefficients) and the problem of ``connection coefficients'' for the GGH polynomials. The asymptotic formulas from the present paper seem promising in the study of the zeros of the $d$-orthogonal polynomials with the Bochner property.

Some other classical issues could be studied as well, such as finding the measures and the corresponding Pearson equations for them. In another project we intend to pursue the connections of the polynomial systems and the corresponding measures with integrable systems such as bi-graded Toda hierarchy and KP-hierarchy. Some of the well known matrix models originate from the special solutions of these hierarchies, corresponding to these polynomial systems or their measures, e.g. the generalized Kontsevich--Penner model, Brezin--Gross--Witten model, etc., see \cite{Al2,Al1, MMS} as well as the Kontsevich~\cite{Kont} model itself. We finally point out that some of these polynomial systems have been studied and used for many years. Apart from the Gould--Hopper polynomials we can mention the Konhauser--Toscano polynomials, whose special cases have been studied and used as early as in 1951 in connection with the penetration and diffusion of the $X$-rays, see~\cite{SF}. These polynomials found more recent applications in the random matrix theory where they have been used in the so-called Muttalib--Borodin biorthogonal ensembles. The latter appeared in the studies of disordered conductors in the metallic regime cf.~\cite{Bor,FW, Mu, Zh}. Also some of the continuous families describe products of Ginibre random matrices, cf.~\cite{AIK, KZh} (again by making use of the hypergeometric and Meijer's G-functions representations). Based on the rich mathematical properties of the $d$-orthogonal polynomials from this paper and in~\cite{Ho3} we hope that their study will be useful also in other problems. However here we do not pursue direct application but rather to give a unified approach to all these important special cases, scattered in the literature, see, e.g., \cite{BCBR1,BCD2, BCZ,GVZ, GH, VZh}.

 The paper is organized as follows. In order to make it independent of~\cite{Ho3}, in Section~\ref{pre} we recall all definitions and statements needed in the main part of the paper. We also give a proof of Hahn's property. Section~\ref{hyp} is the central one; here we derive different formulas for the GGH polynomials in terms of the hypergeometric functions, including some well-known. However our approach is novel; it is based entirely
 on the algebraic construction from \cite{Ho3} and it emphasizes the common origin of all formulas, old and new. Next, in Section~\ref{genf} using the hypergeometric representations, we derive the generating functions for the GGH polynomials. In Section~\ref{MHA} we prove some results of the Mehler--Heine type asymptotics for GGH polynomials. We finish the paper with a number of examples, see Section~\ref{exa}. Besides pure illustrations some of them provide new interpretations, see, e.g., certain cases of ``matching polynomials of graphs''.

\section{Preliminaries} \label{pre}
 To make the present paper self-contained below we briefly recall some of the required notions and results obtained in~\cite{Ho3}.

Given a field ${\mathcal F}$ of characteristic zero, consider the Weyl algebra $W_1$ with coefficients in ${\mathcal F}$ spanned by two generators $Y$, $Z$ subject to the relation $[Z, Y] = 1$, where $[Z, Y]:=ZY - YZ$ is the standard commutator.

Let us introduce some subalgebras of~$W_1$. Set $H = YZ$. Fixing a nonzero polyno\-mial~$R(H)$, put $G = R(H)Z$. The first subalgebra~${\mathcal B}_1\subset W_1$ is, by definition, generated by the elements~$H$,~$G$,~$Y$. These elements satisfy the following relations:
\begin{gather*}
[H, Y]= Y, \qquad [H, G] = - G, \qquad [G, Y]= R(H)(H+ 1) - H R(H-1).
\end{gather*}
For any polynomial $q(G)$ without a constant term\footnote{The constant term contributes only to multiplication of the polynomials, defined below in \eqref{poly}, by a constant. On the other hand, if it vanishes, the formulas and the arguments become simpler.}, define the automorphism of ${\mathcal B}_1$ given by the operator $\sigma_q = e^{\operatorname{ad}_{q(G)}}$.

The images of the generators of ${\mathcal B}_1$ under the automorphism $\sigma_q$ are given below.

\begin{Lemma}\label{lemma-aut}
In the above notation,
\begin{gather*}
\sigma_q(G) = G,\qquad
\sigma_q(H) = H + q'(G)G, \qquad 
\sigma_q Y = Y + \sum\limits_{j=0}^{ld+l-1}\gamma_j(H)G^j.
\end{gather*}
for some polynomials $\gamma_j(H)$.
\end{Lemma}

To move further, we need to introduce an auxiliary algebra ${\mathcal R}_2$ over ${\mathcal F}$ defined by the gene\-ra\-tors~$T$, $T^{-1}$, $\hat{n}$ subject to the relations:
\begin{gather*}
T\cdot T^{-1} = T^{-1} \cdot T=1, \qquad [T, \hat{n}] = T, \qquad \big[T^{-1}, \hat{n}\big] = -T^{-1}.
\end{gather*}
One can easily check that $\big[T, \hat{n} T^{-1}\big] = 1$ which implies that the operators~$T$, $\hat{n} T^{-1}$ determine a~realization of the Weyl algebra~$W_1$.

We can now introduce another non-commutative algebra ${\mathcal B}_2$ as follows. First we define an anti-homomorphism~$b$, i.e., a map $b\colon {\mathcal B}_1 \to {\mathcal R}_2$ satisfying $b(m_1 \cdot m_2) = b(m_2) \cdot b(m_1)$, for each $m_1, m_2 \in {\mathcal B}_1$ given by
\begin{gather*}
b(Y)= T,\qquad b(H) = \hat{n},\qquad b(G) = \hat{n} T^{-1}R(\hat{n}).
\end{gather*}

By definition, the algebra ${\mathcal B}_2$ is the image $ b({\mathcal B}_1)$. With this definition, $b\colon {\mathcal B}_1 \to {\mathcal B}_2$ is an anti-isomorphism and, in particular, $b^{-1} \colon {\mathcal B}_2 \to {\mathcal B}_1$ is well defined.

Observe that we can represent the algebra ${\mathcal B}_1$ on the space ${\mathbb C}[x]$ of polynomials in one variable by realizing $Y$ as the operator~$x$ of multiplication by~$x$, and $Z$ as the operator of differentia\-tion~$\partial_x$.

Consider the polynomial system $\psi(x, n) = x^n$. Then the action of the operators $H$, $G$ on ${\mathbb C}[x]$ is given by $H \to x\partial_x$, $G\to R(x\partial_x) \partial_x$. In the same way, we can represent in ${\mathbb C}[x]$ the algebra~${\mathcal B}_2$ by realizing~$T$ and~$T^{-1}$ as the shift operators acting on a function $f(n)$ as $T^{\pm}f(n) = f(n\pm 1)$. Finally, $\hat{n}$ acts on ${\mathbb C}[x]$ by multiplication by the number~$n$.

Using this notation, we get
\begin{Lemma}\label{bisp}
\begin{gather*}
G \psi(x,n) = \hat{n}R(\hat{n}-1) T^{-1} \psi(x,n),\qquad
H \psi(x,n) = \hat{n} \psi(x,n),\qquad
x \psi(x,n) = T \psi(x,n).
\end{gather*}
\end{Lemma}

 Furthermore, using the operator $q(G)$, we can define another polynomial system $\{P^q_n(x)\}$ as
 \begin{gather} \label{poly}
 P_n^q(x) = e^{q(G)}\psi(x,n)= \sum_{j=0}^{\infty}\frac{q(G)^j\psi(x,n)}{j!}.
 \end{gather}
Notice that, in fact, the above series is always finite since the operator $q(G)$ reduces the degree of any polynomial it is applied to. Denote by $L$ the operator $\sigma_q(H)$ and put $d=\deg R$, $l=\deg q$.

\begin{Theorem} \label{any} In the above notation, the polynomials $P_n^q(x)$ have the following properties:
\begin{enumerate}\itemsep=0pt
\item[$(i)$] They are the eigenfunctions of the differential operator
\begin{gather*} 
L := q'(G)G + x\partial
\end{gather*}
with the eigenvalues $\lambda(n) =n$.

\item[$(ii)$] They satisfy a recurrence relation of the form
\begin{gather*} 
xP^q_n(x) =P^q_{n+1}(x) + \sum_{j=0}^{ld +l -1}\gamma_j(n)P^q_{n-j}(x).
\end{gather*}
\item[$(iii)$] They possess the Hahn's property, i.e., their derivatives are of the same class with a new $ \hat{G} = R(H+1)\partial_x$.
\end{enumerate}
\end{Theorem}

\begin{proof} Proof of (i) and (ii) can be found in~\cite{Ho3}. The proof of (iii) is simple and goes as follows. We have
\begin{gather*}
\partial_x P_n(x) = \sum_{j=0}^{\infty} \frac{\partial_xq(G)^j x^n}{j!}.
\end{gather*}
Notice that $\partial_x G =\partial_x R(H)\partial_x = R(H+1)\partial_x^2$. Hence $\partial_xq(G)^j x^n = n q(\hat{G})x^{n-1}$. This shows that the system $Q_{n-1}(x) = P^{'}_n(x)/n$, $n=1, 2, \ldots$ is the system corresponding to the opera\-tor~$q(\hat{G})$.
\end{proof}

Notice that, similarly to the above, we have earlier realized the abstract construction of ${\mathcal B}_1$ but in terms of difference operators instead of differential ones, acting again on the space of polynomials~${\mathbb C}[x]$~\cite{Ho3}.

Then we define the operators $\tau_{\pm}$ acting on $f \in {\mathbb C}[x]$ by the shift of the argument $\tau_{\pm 1}f(x) = f(x\pm 1)$. The operator~$x$ acts on~$f(x)$ as multiplication by~$x$. We also need the notation
\begin{gather*} \Delta = \tau_{+1} -1,\qquad \nabla = \tau_{-1} - 1, \qquad H = - x\nabla.\end{gather*}
 Finally, put $g = x - H$.
For this realization we use the following polynomials system $\psi(x, n) = (-1)^n(-x)_n = x(x-1)\cdots(x-n+1)$. Here we use the Pochhammer symbol
 \begin{gather} \label{Pocch}
(a)_j = a(a+1)\cdots (a+j-1), \qquad (a)_0 =1.
\end{gather}

\begin{Lemma}\label{dis-bi}
The following identities hold:
\begin{gather*}
H\psi(x,n) = n\psi(x,n), \qquad g\psi(x,n) = T \psi(x,n), \qquad \Delta\psi(x,n) = n T^{-1}\psi(x,n).
\end{gather*}
\end{Lemma}

Exactly as in the continuous case we define the operator $G = R(H)\Delta$. Let $q(G)$ be a~poly\-nomial without a constant term. Then we can define the automorphism $\sigma_q = e^{\operatorname{ad}_{q(G)}}$. Also introduce the new polynomial system $\{P_n^g(x)\}$ given by
\begin{gather*}
P_n^q(x) = e^{q(G)}\psi(x,n).
\end{gather*}
In \cite{Ho3} among other things, we proved the following properties of $\big\{P_n^q(x)\big\}$.

\begin{Theorem}\label{any-dis}
The polynomials $P_n^q(x)$ have the following properties:
\begin{enumerate}\itemsep=0pt
\item[$(i)$] They are eigenfunctions of the difference operator
\begin{gather*}
L := q'(G)G - x\nabla.
\end{gather*}
\item[$(ii)$] They satisfy a recurrence relation of the form
\begin{gather*}
xP^q_n(x) =P^q_{n+1} + \sum_{j=0}^{m}\gamma_j(n)P^q_{n-j},
\end{gather*}
where $m=l$ for $d=0$ and $m=ld+l-1$ for $d>0$.
\item[$(iii)$] They possess the Hahn's property, i.e., their polynomial system
\begin{gather*}
Q_n(x) \stackrel{def}{=} \frac{\Delta P_{n+1}}{n+1},\qquad n=0,1,\ldots,
\end{gather*}
corresponds to $q(\hat{G})$ with a new $ \hat{G} = R(H+1)\Delta$.
\end{enumerate}
\end{Theorem}

\begin{proof} For (i) and (ii) see \cite{Ho3}. The proof of (iii) is similar to the one of Theorem~\ref{any}, using the relation $[\Delta, H] = \Delta$.
\end{proof}

\section{Hypergeometric representations} \label{hyp}

In this section which is central for the present paper, we derive formulas for some of the families of $d$-orthogonal polynomials introduced in~\cite{Ho3} and Section~\ref{pre} in terms of the generalized hypergeometric functions. These families are defined via the operators $q(G) = \rho G^l$, $l \geq 1$.

This property is very important and, in particular, it makes our polynomial systems similar to the classical orthogonal polynomials. Notice that the system of Laguerre polynomials $\big\{L^{(\alpha)}_n(x)\big\}$ corresponds to $G= (x \partial_x +\alpha +1) \partial_x$ and $l=1$, $\rho=1$. The system of Hermite polynomials corresponds to $\partial_x$ and $q(G) = -G^2/2$.

Let us recall the definition of generalized hypergeometric series, cf.~\cite{NIST}. Given a pair of nonnegative integers $(p, q)$, let $\alpha_1, \ldots, \alpha_p$ and $\beta_1, \ldots, \beta_q$ be complex constants and let~$x$ be a~complex variable. (The parameters $\beta_i$ are assumed to be different from non-positive integers.) The series
\begin{gather} \label{def-hyp}
 \pFq{p}{q}{\alpha_1, \ldots, \alpha_p}{\beta_1, \ldots, \beta_q}{x} \stackrel{\rm def}{:=} \sum\limits_{j=0}^{\infty}\frac{(\alpha_1)_j\cdots (\alpha_p)_j}{(\beta_1)_j\cdots(\beta_q)_j}\frac{x^j}{j!},
\end{gather}
 where $(\alpha)_n$ is the Pochhammer symbol~\eqref{Pocch} (also the symbol of the rising factorial), is called a~{\it generalized hypergeometric series}. When it is convergent in some open set its analytic continuation is called generalized hypergeometric function.

 This function satisfies the differential equation
 \begin{gather*}
\left\{\partial_x \prod_{j=1}^{q}( x\partial_x + \beta_j -1) - \prod_{k=1}^{p}( x\partial_x + \alpha_k) \right\}
 \pFq{p}{q}{\alpha_1, \ldots, \alpha_p}{\beta_1, \ldots, \beta_q}{x} = 0.
 \end{gather*}

 We have to point out that the above function does not always exist. However, when one of the parameters $\alpha_i$ is a non-positive integer, the series terminates and the function becomes a~polynomial. In this paper we mainly deal with generalized hypergeometric polynomials. In the remaining cases the series will be obviously convergent, due to the fact that $p<q$ in the gene\-ra\-ting functions and in the Mehler--Heine asymptotic formulas. (For more details consult~\cite{NIST}.) From now on we will omit the word ``generalized'' since there will be no danger of confusion. Notice however that the expression ``the hypergeometric function'' usually refers to $\pFqS{2}{1}{a,b}{c}{x}$.

\subsection{Notation}\label{section3.1}

In what follows we will introduce some notation that will be used to formulate and prove the results. First we will use the shorthand notation
 \begin{gather*}
 \pFq{p}{q}{(\alpha_p)}{(\beta_q)}{x}
 \end{gather*}
 for the hypergeometric series \eqref{def-hyp}. Let us introduce $\alpha_{d+1} = 0$ in an effort to make the notation less clumsy. By $(\alpha_{d+1})$ we denote the vector of parameters $( \alpha_1, \ldots,\alpha_d, \alpha_{d+1})$. Recall that the notation $\Delta(l; \lambda)$ abbreviates the vector
 \begin{gather*}
 \left( \frac{\lambda}{l}, \frac{\lambda +1}{l}, \ldots , \frac{\lambda +l -1}{l}\right),\qquad l \in {\mathbb N},
 \end{gather*}
 of $l$ parameters.

 We will also combine the latter symbol with $( \alpha_{d+1})$ to write $(\Delta(l; \alpha_{d+1}))$ for
\begin{gather*}
 \frac{\alpha_1}{l}, \ldots , \frac{\alpha_1 +l - 1}{l}, \ldots , \frac{\alpha_{d+1}}{l}, \ldots, \frac{\alpha_{d+1} +l - 1}{l},
 \end{gather*}
 in the expressions for the hypergeometric functions.

 We assume that $\rho \neq 0$. Let us put $\eta = l^{d+1}\rho$.

 Let us fix $i\in \{0, 1, \ldots, l-1\}$ and consider the polynomials with $n = ml + i$, $ m = 0, 1, \ldots$. Then the parameters $\Delta(l; -n-\alpha_{d+1})$ can be presented as follows
 \begin{gather*}
\Delta(l; -n-\alpha_{d+1}) = \left(-m - \frac{i + \alpha_1}{l}, -m - \frac{i-1 + \alpha_1}{l}, \ldots, -m - \frac{i + \alpha_1 - l+1}{l},\ldots,\right.\\
\left. \hphantom{\Delta(l; -n-\alpha_{d+1}) =}{} -m - \frac{i + \alpha_{d+1}}{l}, -m - \frac{i-1 + \alpha_{d+1}}{l}, \ldots, -m - \frac{i + \alpha_{d+1} - l+1}{l} \right).
 \end{gather*}
Let us introduce the sets
\begin{gather*}
 S_k(i) = \left\{ \frac{\alpha_k +i-r}{l} +1, \ r= 0,\ldots, l-1\right\}, \qquad k=1, \ldots, d+1
 \end{gather*}
and $ S(i) = \cup_{k=1}^{d+1} S_k(i) $. By $\hat{S}(i)$ we denote the set $S(i)\setminus\{1\}$, 1 is the element, corresponding to $k=d+1$, $r=i$. The elements of $S (i)$ will be denoted by $S_{\beta}(i)$, $\beta= 1,\ldots, (d+1)l$. By~$I$ we denote the set of indexes~$\beta$, which correspond to~$\hat{S}(i)$. Notice that~$I$ depends on~$i$.

We also introduce a notation for the product of Pochhammer symbols:
\begin{gather*}
[\alpha_p]_n \stackrel{\rm def}{=} \prod_{k=1}^{p}(\alpha_k)_n.
\end{gather*}
 Most of the notation is taken from \cite{Is, KLS,LCS, SM}.

\subsection{Continuous GGH polynomials} \label{c-hyp}

Some of the polynomial systems of the previous section have well known representations in terms of generalized hypergeometric functions~\cite{BCD2}, where the authors take the formulas as their definition. Below we will derive the representation from~\cite{BCD2} on the basis of the constructions from~\cite{Ho3}. Our proof will help us to find hypergeometric representations for the other systems studied here, that are not treated elsewhere. The discrete $d$-orthogonal polynomials will be treated in the same manner.

We start with the case when the automorphism $\sigma$ is defined by $G = R(H)\partial_x$ where $R(H)$ is a polynomial. Before that let us factor the polynomial $R(H)$ into $R(H) = \rho\prod\limits_{j=1}^{m}(H +\alpha_j +1)$, $\rho \in{\mathbb C}$, $\rho \neq 0$. Here some of the numbers $\alpha_j$ can be equal. In what follows we will drop the dependence of the $d$-orthogonal polynomials on $R$.

\begin{Theorem} \label{1hyp}
The polynomials $P_n(x)$ have the following hypergeometric representations:
\begin{gather}
	P_n(x) = x^n\, \pFq{d+1}{0}{-n ,(-n -\alpha_d)}{-}{(-1)^{d+1}\rho x^{-1}}, \label{hyp11}\\
	P_n(x) = \rho^n [\alpha_k +1]_n\, \pFq{1}{d}{-n}{(\alpha_d +1)}{-\frac{x}{\rho}}.	\label{hyp12}
\end{gather}
For the second formula the coefficients $\alpha_j$ have to be different from negative integers or zero.
\end{Theorem}

\begin{proof} Notice that
\begin{gather*}
G x^n = n R(H)x^{n-1} = nR(n-1)x^{n-1}.
\end{gather*}
By induction we find that for any $j \in {\mathbb N}$
\begin{gather*}
G^jx^n = \prod_{s=0}^{j-1}(n-s)R(n -1 -s)x^{n-j}.
\end{gather*}
 In terms of the Pochhammer symbol we can write the last formula as
\begin{gather*}
 G^jx^n = (-1)^{j(d+1)}(-n)_j\rho^j[-n - \alpha_d]_jx^{n-j}.
\end{gather*}
The last formula allows to express the polynomials $P_n(x)$ as
\begin{gather}
	P_n(x) = x^n \sum_{j=0}^n\frac{ (-1)^{j(d+1)}(-n)_j [-n -\alpha_d ]_j \rho^jx^{-j}}{j!}\nonumber\\
\hphantom{P_n(x)}{} = x^n \pFq{d+1}{0}{-n,(-n -\alpha_d)}{-}{(-1)^{d+1}\rho x^{-1}}.\label{2hyp}
\end{gather}

To obtain the second expression for the polynomials $P_n(x)$, we transform the coefficients in the first line of \eqref{2hyp} using the formula
\begin{gather}\label{trans}
(-n -\mu)_j =	(-1)^j \frac{(\mu +1)_{n}}{(\mu +1)_{n-j}}.
\end{gather}
Substitute this expression for the parameter values $\mu = 0$, $\alpha_1, \ldots, \alpha_d$ into the sum~\eqref{2hyp}, defi\-ning~$P_n(x)$. As a result, we obtain
\begin{gather*}
 P_n(x) =\rho^n [\alpha_d+1]_n\sum_{j=0}^n \frac{n!}{j!} \frac{ 1}{[\alpha_d +1]_{n-j}} \frac{ \rho^{-n+j} x^{n-j}}{(n-j)!}. 
\end{gather*}
Write the expression $n!/j!$ as $(-1)^{n-j} (-n)_{n-j}$. After changing the summation index $j \rightarrow s = n-j$, we find
\begin{gather*}
P_n(x) = \rho^n
[\alpha_d +1]_n \, \pFq{1}{d}{-n}{(\alpha_d +1)}{-\frac{x}{\rho}}.
\end{gather*}
Of course we need to impose the condition $\alpha_k \neq - 1, -2, \ldots$.
\end{proof}

\begin{Remark} \label{BCD} (1)
This expression coincides with the corresponding one in~\cite{BCD2}. The form of the roots of $R$ was chosen for this purpose as well as to obtain the hypergeometric formula of the generalized Laguerre polynomials~$L^{(\alpha)}_n(x)$.

(2) Formula~\eqref{2hyp} is valid for all values of the coefficients~$\alpha_j$. However if some of the coefficients~$\alpha_j$ is are negative integers, the polynomial system becomes $x^n$ for $n \geq -\alpha_j$ and therefore in such a case it forms a finite system of $d$-orthogonal polynomials.
\end{Remark}

Let us consider the polynomials obtained by the automorphisms $\sigma_q$, where $q(G)$ is some polynomial without constant term. We recall that they are given by
\begin{gather*}
	 P_n^q(x): = e^{q(G)}x^n = x^n + \sum_{k=1}^{\infty}\frac{q^k(G) x^n}{k!}.
\end{gather*}

I don't know if all these $d$-orthogonal polynomials have representations in terms of generalized hypergeometric functions or other special functions. However in the case when $q(G) = \rho G^l$ they have. We are going to present the formula. Let us fix $q(G) = \rho G^l$. Finally we will drop the dependence on~$q(G)$ as there is no danger of confusion.

\begin{Theorem}\label{ch} For $q(G) = \rho G^l$ the polynomials $P_n(x)$ have the following representation in hypergeometric functions:
\begin{gather}\label{Gl}
	P_n(x) = x^n\, \pFq{dl+l}{0}{ (\Delta(l; -n-\alpha_d))} {-}{\left(\frac{(-1)^{d+1}\eta}{x}\right)^l}.
\end{gather}
\end{Theorem}
\begin{proof} Again we use the formula for $P_n(x)$:
\begin{gather*}
P_n(x) = \sum_{j=0}^{\infty} \frac{G^{lj} x^n }{j!}.
\end{gather*}
We know that
\begin{gather*}
G^{lj} x^n = \prod_{s=0}^{lj-1}(n-s)R(n -1 -s)x^{n-lj}.
\end{gather*}

Using that $\alpha_{d+1} =0$ we present the coefficient at $x^{n-lj}$ of $G^{lj}x^n$ in the form
\begin{gather} \label{Glj}
 \eta^{lj} (-1)^{lj(d+1)} \prod_{r= 0}^{l-1} \left[\frac{-n-\alpha_{d+1}+r}{l}\right]_j.
\end{gather}
Then the polynomials $P_n(x)$ are given by
\begin{gather*}
P_n(x) = x^n + \sum_{j=1}^n\frac{(G^l)^j x^n}{j!} = x^n \sum_{j=1}^n\eta^{lj} (-1)^{lj(d+1)} \prod_{r= 0}^{l-1} \left[\frac{-n-\alpha_{d+1}+r}{l}\right]_j\frac{x^{-lj}}{j!} \nonumber\\
\hphantom{P_n(x)}{} = x^n \sum_{j=0}^{n} \prod_{r=0}^{l-1} \left[\frac{-n-\alpha_{d+1}+r}{l}\right]_j \left(\frac{(-1)^{d+1}\eta}{x}\right)^{lj}\frac{1}{j!}.
\end{gather*}
This proves \eqref{Gl}.
\end{proof}

\begin{Remark}\label{her}The formula for the $d$-orthogonal polynomials, corresponding to $q(G)= \rho G^l$, in the above theorem resembles the representation of Hermite polynomials in hypergeometric functions, see, e.g.,~\cite{NIST}. Notice that we don't need to sum up to~$n$ but only up to $\big\lceil{\frac{n}{l}}\big\rceil $ as the next terms are~0.
\end{Remark}

As in the case $l=1$ we are going to find a second representation.

\begin{Theorem}\label{th-new-h}The polynomials $P_n(x)$, where $n= ml +i$, have the following hypergeometric representation
\begin{gather} \label{new-h}
P_n(x) =\eta^{ml} x^i \prod_{\beta \in I}\big(\hat{S}_{\beta}(i)\big)_m \,\pFq{1}{ld +l}{-m}{\big(\hat{S}(i)\big)}{- \left[\frac{x}{\eta}\right]^l}.
\end{gather}
\end{Theorem}
\begin{proof} Notice that \eqref{Gl} can be written in the form
\begin{gather*}
	P_n(x) = x^{ml +i} \,\pFq{dl+l}{0}{ (-S(i) +1 -m)} {-}{\left(\frac{(-1)^{d+1}\eta}{x}\right)^l}.
 \end{gather*}
 Let us put $ y = \big(\frac{x}{\eta}\big)^{l} $ and consider the polynomials $P_m^{(i)}(y)$ defined by $P_n(x) = x^i \eta^{ml}P_m^{(i)}(y)$. We use that one of the elements of the set $S(i) $ is~$1$, which shows that they have the representation~\eqref{hyp12} from Theorem~\ref{1hyp}
 \begin{gather*}
 P_m^{(i)}(y) = y^{m} \,\pFq{dl+l}{0}{-m, (-\hat{S}(i) +1 -m)} {-}{\frac{(-1)^{ld+l}}{y}}.
 \end{gather*}
 They are exactly in the form of~\eqref{hyp11}. This means that they can be written as
 \begin{gather*}
 P_n^{(i)}(y)= \eta^{ml} \prod_{\beta \in I}(S_{\beta}(i))_m \,\pFq{1}{ld +l-1}{-m}{\big(\hat{S}(i)\big)}{- y}.
 \end{gather*}
 Returning to the polynomials $P_n(x)$ we obtain the formula~\eqref{new-h}.
\end{proof}

\begin{Corollary}\label{new-hypF2} The polynomials $P_n(x)$ have also the representation
\begin{gather*}
P_n(x)= \eta^{ml} \prod_{\beta \in I }(S_{\beta}(i))_m \,\pFq{2}{ld +l-1}{-m, 1}{(S(i))}{- \left[\frac{x}{\eta}\right]^l}.
\end{gather*}
\end{Corollary}

\begin{proof} Use that $S(i) = \hat{S}(i) \cup \{1\}$. Then the general formula
\begin{gather*}
\pFq{p}{q}{(\alpha_p)}{(\beta_q)}{x} =\pFq{p+1}{q+1}{(\alpha_p), 1}{(\beta_q), 1}{x}
\end{gather*}
applied to \eqref{new-h} gives the result.
\end{proof}

We see that the GGH polynomial system $P_n(x)$ can be split into~$l$ families of $d$-orthogonal polynomials with Bochner's property exactly as Hermite polynomials split into~2 families in terms of Laguerre polynomials. More precisely

\begin{Corollary}\label{lag-typ}
The polynomial system $P_n(x)$ consists of $l$ families $P_n(x) = C_i x^i P_m^{(i)}(y)$, $i=0, \ldots, l-1$, where
\begin{gather*}
y = \left(\frac{x}{\eta}\right)^{l}, \qquad C_i = \eta^{ml} \prod_{\beta \in I} (S_{\beta}(i) )_m.
\end{gather*}
The family $P_m^{(i)}(x)$ is the system of $d$-orthogonal polynomials with Bochner's property corresponding to $G= \prod\limits_{\beta \in I}(H +S_{\beta})\partial$.
\end{Corollary}

While the above statement does not need a proof it is worth to point out that it is a precise formulation of Remark~7.4. in~\cite{Ho3}. In \cite[Example~7.3]{Ho3} we demonstrated that such a splitting of the family exists for the simplest case of $q(G) = G^2$, where $G = (x\partial_x + \alpha +1)\partial_x$ without using hypergeometric functions.

\subsection{Discrete GGH Polynomials}

The discrete polynomial systems also have representation in terms of generalized hypergeometric functions. Below we will derive them using the approach that was exploited for continuous $d$-orthogonal polynomials. These representations are new except for the cases of Charlier and Meixner polynomials.

We again factor the polynomial $R(H)$ into $R(H) = \rho\prod\limits_{j=1}^{m}(H+\alpha_j +1)$.

\begin{Theorem}\label{hyp-d}The discrete $d$-orthogonal polynomials $P_n(x)$ have the following hypergeometric representations:
\begin{gather}
	P_n(x)	= (-1)^n(-x)_n \,\pFq{d+1}{1}{-n , (-n -\alpha_d )}{ x-n+1}{(-1)^{d+1}\rho },\label{2for}
\end{gather}
and
\begin{gather}	
	P_n(x)	= [\alpha_d+1]_n \,\pFq{2}{d}{-n, -x }{ (\alpha_d +1)}{\rho }.\label{4hyp}
\end{gather}
\end{Theorem}

\begin{proof} First we transform the formula for $P_n(x)$ to get rid of the difference operators. We are going to exploit again the formula
\begin{gather*}
P_n(x)= \sum\limits_{j=0}^{\infty}\frac{G^j\psi(x,n)}{j!} \qquad \text{with} \quad \psi(x,n) = (-1)^n (-x)_n.
\end{gather*}
Using that
\begin{gather*}
G \psi(x,n) = n R(n-1)\psi(x, n-1),
\end{gather*}
we obtain
\begin{gather*}
G^j \psi(x,n) = \rho^j (-1)^{n-j}(-x)_{n-j} (-1)^j (-n)_{j} (-1)^j [-\alpha_d - n]_{j}.
\end{gather*}
Let us use the formula
\begin{gather} \label{trans1}
(-1)^{n-j}(-x)_{n-j} = \frac{(-x)_n (-1)^n}{(x-n +1)_j}.
\end{gather}
We obtain that
\begin{gather*}
G^j \psi(x,n) = (-1)^n(-x)_n \frac{ \rho^j (-1)^{dj +j} (-n)_{j} [-\alpha_d - n]_{j}}{(x-n +1)_j}.
\end{gather*}
For the polynomials $P_n(x)$ this gives the expression
\begin{gather*}
	P_n(x)= (-1)^n 	(-x)_n \sum_{j=0}^n\frac{\rho^j (-1)^{(d+1)j} (-n)_j [-\alpha_d - n]_{j}}{(x-n +1)_j}{(x-n+1)_j j!},
\end{gather*}
which is \eqref{2for}.

The second hypergeometric expression can be obtained as in the continuous case. Using again~\eqref{trans} we find
\begin{gather*}
P_n(x)=	\prod_{k=1}^m (\alpha_k+1)_n \sum_{j=0}^n \frac{n! }{(n-j)!} \frac{(-1)^{n-j}\rho^j(-x)_{n-j}}{j! [-\alpha_d - n]_{n-j} }{(x-n +1)_j}.
\end{gather*}

After standard manipulations that we used for the continuous $d$-orthogonal polynomials we come to the second formula~\eqref{4hyp}.
\end{proof}

We can obtain hypergeometric representations for the class of discrete $d$-orthogonal polynomials, constructed via $q(G)$, where $G = R(H)\Delta$. Again we will treat the case when $q(G) = \rho G^l$. Let us put
\begin{gather*} \eta_1 = \frac{(-1)^{(d+1)l}\eta}{l} .\end{gather*}

\begin{Theorem}\label{dh}
In the case when $q(G) = \rho G^l$ the polynomials $P_n(x)$ have the following representation in hypergeometric functions
\begin{gather} \label{hyp-d1}
	P_n(x) = (-1)^n(-x)_n \,\pFq{dl+l}{l}{\Delta(l; -n), (\Delta(l; -n-\alpha))} {\Delta(l; x-n +1)}{ \eta_1^l}.
\end{gather}
\end{Theorem}
\begin{proof} The proof needs a few changes in comparison with the continuous case but otherwise is straightforward. We are going to use again the series defining the polynomials
\begin{gather*}
	P_n(x) = \sum_{j=0}^{\infty}\frac{G^{jl} \psi(x,n)}{j!}.
\end{gather*}
A formula, similar to the one in the continuous case holds
\begin{gather*}
G^{lj} (-1)^n(-x)_n = \prod_{s=0}^{lj-1} (n-s)R(n- 1 -s)(-x)_{n-lj}(-1)^{n-lj}.
\end{gather*}
For the term $(-x)_{n-lj}$ we use \eqref{trans1} with $lj$ instead of $j$
\begin{gather*}
(-1)^{n-lj} (-x)_{n-lj} = \frac{(-1)^{n}(-x)_n}{(x-n +1)_{lj}}.
\end{gather*}
The factor $(x-n +1)_{lj}$ can be presented as indicated in Section~\ref{c-hyp}
\begin{gather*}
(x-n +1)_{lj} = l^{lj}\prod_{r=0}^{l-1}\left(\frac{x-n +r +1}{l}\right)_{j}.
\end{gather*}
This gives
\begin{gather} \label{n-lj}
(-1)^{n-lj}(-x)_{n-lj} = \frac{(-1)^n(-x)_n}{ l^{lj}\prod\limits_{r=0}^{l-1}\left(\frac{x-n +r +1}{l}\right)_j}.
\end{gather}

Using \eqref{Glj} for $\prod\limits_{s=0}^{lj-1} (n-s)R(n- 1-s) $ and expressing the last factor by~\eqref{n-lj} in the sum for the polynomials we obtain the desired formula~\eqref{hyp-d1}.
\end{proof}

Again we can find a second formula for the $d$-orthogonal polynomials, corresponding to \smash{$q(G)= \rho G^l$}. We use the notations~\eqref{c-hyp} from Section~\ref{section3.1}. With this notation we have

\begin{Theorem} \label{new-dh}
The polynomials $P_n(x)$, where $n= ml +i$ have the following hypergeometric representation
\begin{gather} \label{new-d}
P_{ml +i}(x) = (-1)^{mdl +1} C(i) (-x)_i\prod_{k=1}^{ld +l-1} \,\pFq{l}{ld +l -1}{-m, \Delta(l; -x+i -1)}{ \big(\hat{S}(i)\big)}{-\eta_1^{-l}},
\end{gather}
where $C(i)$ was defined in Corollary~{\rm \ref{lag-typ}}.
\end{Theorem}
\begin{proof}
We start with the formula
\begin{gather*}
P_n(x) = \sum_{j=0}^{\infty} \frac{\prod\limits_{s=0}^{lj-1} (n-s)R(n- 1 -s) (-1)^{n-lj}(-x)_{n-lj}}{j!} .
\end{gather*}
 We transform the expression
 \begin{gather*}
 \prod_{s=0}^{lj-1} (n-s)R(n-1-s)
 \end{gather*}
 using \eqref{Glj} to obtain
 \begin{gather*}
 (-1)^{lj}\eta^{lj} (-m)_j\prod_{\beta \in I} (-m -S_{\beta}(i))_j.
 \end{gather*}
 We further transform the last expression using \eqref{trans} into
 \begin{gather} \label{prod}
 (-1)^{l(d+1)j}\eta^{lj}\frac{ (-1)^j m!}{(m-j)!}\frac{\prod\limits_{\beta \in I} (S_{\beta}(i))_m}{\prod\limits_{\beta \in I} (S_{\beta}(i))_{m-j}}.
 \end{gather}

 Next we transform the factor $(-x)_{n-lj}$ as follows. First
\begin{gather*}
(-1)^{n-jl}(-x)_{n-lj} = (-1)^{i} (-x)_i (-1)^{ml - jl}(-x+i)_{lm-lj}.
\end{gather*}
Then we present the last factor in the right-hand side of the last formula as
\begin{gather*}
 (-x+i)_{lm-lj} = l^{(m-j)l}\prod_{r=0}^{l-1}\left(\frac{-x+i -r}{l}\right)_{m-j}.
\end{gather*}

 At the end plugging the last formula and \eqref{prod} into the sum for $P_n(x)$ and changing the summation index $j \to m-j$ we obtain
\begin{gather*}
P_{ml +i} = C(i) (-x)_i \sum_{j=0}^{m} \frac{(-m)_j \prod\limits_{r=0}^{l-1}\left(\frac{-x +i -r}{l}\right)_j}{\prod\limits_{\beta \in I} (S_{\beta})_j} \frac{\big[{-}\eta_1^{-l}\big]^j}{j!},
\end{gather*}
which is \eqref{new-d}.
\end{proof}

\section{Generating functions} \label{genf}

In this section we will find generating functions for all $d$-orthogonal polynomials for which we have earlier obtained hypergeometric representations. Generating functions for the continuous $d$-orthogonal polynomials, corresponding to $q(G) = G$, can be found in~\cite{BCD2}.

\subsection{Continuous GGH Polynomials} \label{contGH}

 In what follows we use a different normalization of the $d$-orthogonal polynomials. Namely, we introduce the polynomials
 \begin{gather} \label{pure}
 Q_n(x) := x^{i} \,\pFq{1}{ld +l-1}{-m}{\big(\hat{S}(i)\big)}{\left[\frac{x}{(-1)^{d+1}\eta\rho}\right]^l}.
 \end{gather}
 They differ from the polynomials $P_n(x)$, given by~\eqref{new-h} by a multiplicative constant. We also assume that $((-1)^{d+1}\eta)^l = -1$ which can be achieved by rescaling of $x$ together with multiplication by a suitable constant. Our first result is as follows.

 \begin{Theorem}\label{gen-c} For a given positive integer $i$, the function $\Phi_i(x,t)$ defined as
 \begin{gather} \label{gen2-c}
 \Phi_i(x, t) := (tx)^i e^{t^l}\, \pFq{0}{ld+l-1}{-}{\big(\hat{S}(i)\big) }{(xt)^l},
 \end{gather}
 generates the polynomial system $\{Q_{lm +i}(x)\}_{m=0}^\infty$ by means of the formula
 \begin{gather*}
 \Phi_i(x, t)=\sum\limits_{m=0}^{\infty} \frac{ t^{ml +i}}{m!} Q_{ml +i}(x).
 \end{gather*}
 \end{Theorem}
 \begin{proof}
 Consider the series
\begin{gather*}
\sum\limits_{m=0}^{\infty} \frac{ t^{ml +i}}{m!} Q_{ml +i}(x)= t^ix^i\sum\limits_{m=0}^{\infty} \frac{ t^{ml}}{m!} \,\pFq{1}{ld +l-1}{-m}{\big(\hat{S}(i)\big) }{-x^l}.
\end{gather*}
One can present the sum in the right-hand side of the latter equation in the form of a double series
 \begin{gather*}
 \sum\limits_{m=0}^{\infty} \frac{ t^{ml +i}}{m!} Q_{ml +i}(x) = \sum\limits_{m=0}^{\infty} \frac{ t^{ml}}{m!} \sum\limits_{j=0}^{m} \frac{ (-m)_j }{\prod\limits_{\beta \in I} (S_{\beta}(i))_j} \frac{\big({-}x^l\big)^j}{j!}.
 \end{gather*}
 Changing the order of summation in the double series, we obtain
 \begin{gather*}
 \sum\limits_{j=0}^{\infty} \frac{ \big({-}x^l\big)^j }{\prod\limits_{\beta \in I} (S_{\beta}(i))_j j!} \sum\limits_{m=j}^{\infty} \frac{(-m)_j t^{ml}}{m!}.
 \end{gather*}
 It is easy to see that
 \begin{gather*}
 \frac{(-m)_j} {m!} = \frac{(-1)^j}{(m-j)!}.
 \end{gather*}
 Hence after introducing a new index $s = m-j$ of summation, the double series becomes
 \begin{gather*}
 \sum\limits_{j=0}^{\infty} \frac{ (xt)^{lj} }{\prod\limits_{\beta \in I}(S_{ \beta}(i))_j j!} \sum\limits_{s=0}^{\infty} \frac{t^{sl}}{s!}.
 \end{gather*}
 This gives for the double sum
 \begin{gather*}
 \pFq{0}{ld+l-1}{-}{ \big(\hat{S}(i)\big) }{(xt)^l} e^{t^l},
 \end{gather*}
 which implies \eqref{gen2-c} for the function $\Phi_i(x,t)$.
\end{proof}

\begin{Remark}\label{GBess}
We notice that the hypergeometric functions without upper parameters appear in a quite different bispectral problem. In \cite{BHY2} we defined the functions $\Psi_{\beta}(x,z)$, which are solutions of an equation of the form
\begin{gather*}
x^{-N} (\theta - \beta_1)\cdots (\theta - \beta_N) \Psi_{\beta}(x,z) = z^N\Psi_{\beta}(x,z),
\end{gather*}
where $\theta = x\partial_x$ and $\beta_j \in {\mathbb C}$. We called them generalized Bessel functions. As one of the referees kindly informed me these functions have been studied by P.~Delerue in~\cite{Del}, and are called hyper-Bessel functions. They are expressed in terms of the hypergeometric functions without upper parameters (in~\cite{BHY2} we used Meijer's $G$-functions). Through these functions we were able to find non-trivial bispectral operators of any rank. It is interesting to understand if this is a~mere coincidence or the reasons are deeper. I hope that such a connection exists and in this case it would be useful in the studies of Darboux transformation of the generalized Gould--Hopper polynomials. Even in the case of $l=1$ (for which the same formula was found by different methods in~\cite{BCD2}, see also below Corollary~\ref{l=1}) the connection deserves attention.
\end{Remark}

From the above formulas we can write a generating function for the entire family $Q_n$. Let us define the function
\begin{gather*}
\Phi(x,t) = \sum\limits_{i=0}^{l-1} \Phi_i(x, t) = e^{t^l}\sum_{i=0}^{l-1} (xt)^i \cdot \pFq{0}{ld+l-1}{-}{\big(\hat{S}(i)\big)}{(xt)^l}.
\end{gather*}

\begin{Corollary}\label{gen-nc} The function $\Phi(x,t) $ is a generating function for the polynomials $Q_n(x)/ \lceil{n/l} \rceil !$:
\begin{gather*}
\sum\limits_{n=0}^{\infty} \frac{t^n}{ \lceil{n/l} \rceil !} Q_n(x) = \Phi(x,t).
\end{gather*}
\end{Corollary}

The proof is obvious and we omit it.

\begin{Remark}\label{Bren}
Polynomial systems that have a generating function of the form
\begin{gather*}
\sum_{n=0}^{\infty} P_n(x) t^n = A(t)B(xt).
\end{gather*}
are called Brenke polynomials, see, e.g.,~\cite{BCBR1}. The last corollary shows that the GGH polynomials are Brenke polynomials.
\end{Remark}

It deserves to write separately the formula for the case $l=1$.
\begin{Corollary}[\cite{BCD2}]\label{l=1}
When $l=1$ we have
\begin{gather*}
\sum\limits_{n=0}^{\infty} \frac{t^n}{n!} Q_n(x) = \pFq{0}{d}{-}{(\alpha_d +1)}{xt} e^{t}.
\end{gather*}
\end{Corollary}

We are going to obtain a second formula based on a theorem from \cite{ Sr1}. Let us formulate the corresponding result explicitly in a slightly less general form that suffices for our purposes.
 \begin{Proposition}\label{pr-sri}
 Let $a \in {\mathbb C}$, $-a \notin {\mathbb N}$ and let $\alpha_1, \ldots, \alpha_p$, $\beta_1, \ldots, \beta_q$ be complex numbers such that the hypergeometric function
 \begin{gather*}
 \pFq{p+1}{q +1}{-n, (\alpha_p)}{a +1, (\beta_q)}{x}
 \end{gather*}
 be well defined. Then the following formula holds
 \begin{gather}
 \sum\limits_{n=0}^{\infty} \binom{a +n}{n}\, \pFq{p+1}{q+1}{-n, (\alpha_p)}{a+1, (\beta_q)}{x}t^n = \frac{1}{(1-t)^{a+1}} \, \pFq{p}{q}{ (\alpha_p)}{(\beta_q)}{\frac{xt}{1-t}}. \label{sri}
 \end{gather}
 \end{Proposition}

See \cite{Sr1} for a simple proof.

We again use the $d$-orthogonal polynomials $Q_n(x)$ from \eqref{pure} as well as the convention $\eta = 1$.

\begin{Theorem}\label{gen-c+} The function $G_i$ given by
 \begin{gather*}
 G_i(x, t) = \frac{(tx)^i}{1-t^l} \, \pFq{1}{ld+l-1}{1}{\big(\hat{S}(i)\big) }{\frac{(xt)^l}{1-t^l}},
 \end{gather*}
 generates the polynomials $Q_{lm +i}(x)$, $m= 0, 1, \ldots$ in the following way
 \begin{gather*}
G_i(x, t) = \sum\limits_{m=0}^{\infty} t^{ml +i} Q_{ml +i}(x).
 \end{gather*}
 \end{Theorem}
 \begin{proof} Let us multiply the polynomials $Q_{lm +i}$ by $t^{ml+i}$ and sum. Thus we obtain
 \begin{gather*}
 \sum\limits_{m=0}^{\infty} t^{ml +i} Q_{ml +i}(x)= t^ix^i\sum\limits_{m=0}^{\infty} t^{ml} \,\pFq{2}{ld +l}{-m, 1}{ \big(\hat{S}(i) \big),1}{x^l}.
 \end{gather*}
 Here we have used that
 \begin{gather*}
 \pFq{p}{q}{ (\alpha_p)}{(\beta_q)}{x} = \pFq{p+1}{q+1}{ (\alpha_p), \mu}{(\beta_q) , \mu }{x}, \qquad \mu \neq 0.
 \end{gather*}

 We apply the above cited formula \eqref{sri} from~\cite{Sr1} with $a=0$ to obtain
 \begin{gather*}
 \sum\limits_{m=0}^{\infty} t^{ml+i} Q_{ml +i}(x) = \frac{(tx)^i}{1-t^l} \, \pFq{1}{ld+l-1}{1}{ \big(\hat{S}(i)\big)}{\frac{(xt)^l}{1-t^l}}.\tag*{\qed}
 \end{gather*}\renewcommand{\qed}{}
\end{proof}

From this theorem we can write a generating function for all polynomials $Q_n(x)$.

\begin{Corollary}\label{ggen1}
 A generating function for $Q_n(x)$ is given by
 \begin{gather*}
 \sum\limits_{n=0}^{\infty} t^n Q_n(x) = \sum\limits_{i=0}^{l-1} \frac{(tx)^i}{1-t^l} \,\pFq{1}{ld+l-1}{1}{ \big(\hat{S}(i)\big)}{\frac{(xt)^l}{1-t^l}}.
 \end{gather*}
 \end{Corollary}
\begin{proof} Just sum up the generating functions $\Phi_i(x,t)$ for $i = 0, \ldots, l-1$ and replace $n$ by $ml +i$.
\end{proof}

 Notice that the coefficients $\hat{S}(i)$ depend on $i$, which makes it difficult to obtain a better formula in general. However when $l=1$ the above expression gives

\begin{Corollary}\label{ggen2}
 The $d$-orthogonal polynomials defined in terms of $q(G) = \rho G$ have the following generating function
 \begin{gather*}
 \sum\limits_{n=0}^{\infty} t^n P_n(x) = \frac{1}{1-t} \,\pFq{1}{d}{1}{(\alpha_d +1)}{\frac{x t }{1-t}}.
 \end{gather*}
 \end{Corollary}

 \subsection{Discrete GGH polynomials}
In the discrete case there is nothing special. We follow the arguments for the continuous case. However we will keep the coefficient $\rho$. We again put $n = ml +i$ and fix~$i$. We use the following modification of the polynomials \eqref{new-d}
\begin{gather*}
Q_n(x) = (-1)^i (-x)_i\, \pFq{1+l}{ld +l-1}{-m, \Delta(l; -x+i -1)}{ \big(\hat{S}(i)\big)}{\eta_1^l}, 
\end{gather*}
which differ from $P_n(x)$ only by a multiplicative constant.

\begin{Theorem} 
The function $\Phi_i$ given by
\begin{gather*}
\Phi_i(x, t) = (-1)^j (-x)_i t^i e^{t^l} \, \pFq{l}{ld +l-1}{ \Delta(l; -x+i)}{ \big(\hat{S}(i)\big)}{ -(\eta_1t)^l}.
\end{gather*}
is a generating function for the polynomials $Q_{lm +i}(y)$ in the sense that
\begin{gather*}
\sum\limits_{m=0}^{\infty} \frac{t^{ml +i}}{m!}Q_{ml +i}(x) = \Phi_i(x,t).
\end{gather*}
\end{Theorem}
\begin{proof} We write the defining series as
\begin{gather*}
 (-1)^i (-x)_i\sum\limits_{m=0}^{\infty} \frac{t^{ml +i}}{m!} \sum\limits_{j=0}^{m} \frac{(-m)_j ( \Delta(l; -x+i))_j \eta_1^{jl} }{\prod\limits_{\beta \in I} (S_{\beta}(i) )_jj!}.
\end{gather*}
In the right-hand side we make the following transformations. We first change the order of the summation and then introduce a new summation variable $m \to s = m-j$. We obtain
\begin{gather*}
(-1)^i(-x)_i t^i \sum\limits_{j=0}^{m} \frac{ ( \Delta(l; -x+i))_j \big[\eta_1^{l}]^j }{ \prod\limits_{\beta \in I} (S_{\beta}(i) )_jj!} \sum\limits_{s=0}^{\infty} \frac{(-s -j)_jt^{(s+j)l} }{(s+j)!}.
\end{gather*}
Notice that
\begin{gather*}
\frac{(-s -j)_j}{(s+j)!} = \frac{(-1)^j}{s!}.
\end{gather*}
This gives
\begin{gather*}
\Phi_i(x, t) =(-1)^i(-x)_i t^i \sum\limits_{j=0}^{m} \frac{ \eta_1^{lj} (-1)^j t^{lj }}{\prod\limits_{\beta \in I} (S_{\beta}(i) )_jj!} \sum\limits_{s=0}^{\infty} \frac{t^{sl} }{s!},
\end{gather*}
which is exactly the desired formula.
\end{proof}

As in the continuous case we are going to derive a second formula.

 \begin{Theorem}
 The function $G_i$ given by
 \begin{gather*}
 G_i(x, t) = \frac{t^i(-1)^i(-x)_i}{1-t} \,\pFq{1 +l}{ld +l-1}{ \Delta(l; -x+i), 1}{\big(\hat{S}(i)\big)}{-\left[\frac{(-1)^{d}}{\rho}\right]^l}
 \end{gather*}
 generates the polynomials $Q_{lm +i}(y)$ in the sense that
 \begin{gather*}
 G_i(x, t) = \sum\limits_{m=0}^{\infty} t^mQ_{ml+ i}(x).
 \end{gather*}
 \end{Theorem}

\begin{proof} We are going to use again Proposition~\ref{pr-sri}. It is obvious that we need to put $a=0$. We have
 \begin{gather*}
 \sum\limits_{m=0}^{\infty} t^mQ_{ml+ i}(x) = (-1)^i (-x)_i \sum\limits_{m=0}^{\infty} t^m \;\pFq{2+l}{ld +l}{-m, \Delta(l; -x+i -1), 1}{ \big(\hat{S}(i)\big), 1 }{\eta_1^l}.
 \end{gather*}
 Then~\eqref{sri} gives
 \begin{gather*}
(-1)^i (-x)_i \sum\limits_{m=0}^{\infty} t^{ml+ i} \;\pFq{2+l}{ld +l}{-m, \Delta(l; -x+i -1), 1}{\big(\hat{S}(i)\big), 1}{\eta_1^l} \\
\qquad{} =\frac{(-1)^i(-x)_i t^i}{1- t^l} \,\pFq{1+l}{ld +l-1}{ \Delta(l; -x+i -1), 1}{ \big(\hat{S}(i)\big)}{\frac{\eta_1^l}{1-t^l}}.\tag*{\qed}
 \end{gather*}\renewcommand{\qed}{}
\end{proof}

As a trivial corollary
 again we can write a generating function for all polynomials $Q_n(x)$.
 \begin{Corollary}\label{dgen}
 A generating function for $Q_n(x)$ is given by
\begin{gather*}
 \sum\limits_{n=0}^{\infty} t^n Q_n(x) = \sum\limits_{i=0}^{l-1} \frac{(-1)^i(-x)_i t^i}{1- t^l} \, \pFq{1+l}{ld +l-1}{\Delta(l; -x+i -1), 1}{\big(\hat{S}(i)\big) }{\frac{\eta_1^l}{1-t^l}}.
 \end{gather*}
 \end{Corollary}

Finally for $l=1$ we get a better formula
 \begin{Corollary}\label{ggen3}
 The polynomials defined in terms of $q(G) = \rho G$ have the following generating function
\begin{gather*}
 \sum\limits_{n=0}^{\infty} t^n Q_n(x) = \frac{1}{1-t} \,\pFq{2}{d}{ x, 1}{(\alpha_d+1)}{ \frac{t\rho}{1-t}}.
 \end{gather*}
 \end{Corollary}

\section{Mehler--Heine type formulas} \label{MHA}

In this section, we are going to consider only the continuous $d$-orthogonal polynomials and, without loss of generality, we will assume that $\rho l^{d+1} = 1$. (This assumption will make the formulas simpler.)
There are many ways to write down the Mehler--Heine type formulas depending on the normalization of the polynomials of which we choose only the simplest one.

Let us start with the case of continuous $d$-orthogonal polynomials corresponding to $q(G) = G$ and use the polynomials $Q_n(x)$ given by~\eqref{pure}.

\begin{Theorem}\label{1MH} For the $d$-orthogonal polynomials obtained from $q(G) = G$, the following Mehler--Heine type formula holds
\begin{gather} \lim\limits_{n\to \infty}Q_n(x/n) = \pFq{0}{d}{-}{(\alpha_d +1)}{ x}.\label{MH1}
\end{gather}
\end{Theorem}

\begin{proof} Notice that in the case $l=1$, equation \eqref{sri} can be written as
 \begin{gather*}
 Q_n(x) = \pFq{1}{d}{-n}{(\alpha_d+1)}{-x}.
 \end{gather*}
Then we use the formula
 \begin{gather*}
 \lim\limits_{\lambda \to \infty}\pFq{p +1 }{q}{(\alpha_p), a\lambda}{(\beta_q)}{ \frac{x}{\lambda}} = \pFq{p}{q}{(\alpha_p)}{(\beta_q)}{a x},
 \end{gather*}
with $\lambda = n$, $a =-1$, see, e.g., \cite[p.~5]{KLS}. This immediately gives~\eqref{MH1}.
\end{proof}

\begin{Remark}\label{MH}
 This theorem is proved in \cite{VAss}. We present it here as an illustration of the results of the present paper. Also the Mehler--Heine formula for general GGP follows from it.

As one of the referees kindly pointed to me, the polynomials $Q_n(x)$ considered here all are Jensen polynomials. This means that they are associated to an
entire function $\varphi(x) = \sum\limits_{n=0}^{\infty} \gamma_n \frac{x^n}{n!}$ in the following way
\begin{gather*}
Q_n(x) = \sum_{j=0}^{n} \binom{n}{j} \gamma_j x^j.
\end{gather*}
Also the function $e^t \varphi(xt)$ is their generating function:
\begin{gather*}
e^t \varphi(xt) = \sum_{n=0}^{\infty} Q_n(x) \frac{t^n}{n!},
\end{gather*}
 see, e.g., \cite{CC} for a contemporary reference to properties of Jensen polynomials that we refer to here. In our case
 \begin{gather*}
 \varphi(x) = \pFq{0}{d}{-}{(\alpha_d +1)}{ x}.
 \end{gather*}
By the properties of the Jensen polynomials we have
\begin{gather*}
\lim\limits_{n \to \infty} Q_n(x/n) = \varphi(x) = \pFq{0}{d}{-}{(\alpha_d +1)}{ x}
\end{gather*}
locally uniformly as was found by Jensen himself in 1913.

 Observe that the hypergeometric function
\begin{gather*}
 \varphi(x) = \pFq{0}{d}{-}{(\alpha_d+1)}{x}
\end{gather*}
appears once again naturally as in the formulas for the generating functions for the polynomial system $\{Q_n(x)\}$. In both cases this is connected to the fact that they are Jensen polynomials.
\end{Remark}

Now consider the polynomials obtained by the automorphisms $\sigma_q$, where $q(G)$ is some polynomial. We recall that they are given by
\begin{gather*}
P_n(x): = e^{q(G)}x^n = x^n + \sum_{k=1}^{\infty}\frac{q^k(G) x^n}{k!}.
\end{gather*}
We will restrict ourselves to the case $q(G)= G^l$ for which the corresponding hypergeometric representation was obtained in Section~\ref{hyp}. Presenting $n$ as $n = ml + i$ and using~\eqref{pure} for the polynomials~$Q_n(x)$, we get
 \begin{gather*}
 Q_n(x) = x^i \,\pFq{1}{ld +l-1}{-m}{ \big(\hat{S}(i)\big)}{\big[(-1)^{d+1}x\big]^l}.
 \end{gather*}

\begin{Theorem}\label{2MH} For the $d$-orthogonal polynomials obtained via the automorphisms $\sigma_q$ with\linebreak \smash{$q(G) = G^l$}, the following Mehler--Heine type formula holds
\begin{gather} \lim\limits_{m\to \infty} m^{i/l} Q_n\big(x/m^{1/l}\big) = x^i \, \pFq{0}{ld +l-1}{-}{\big(\hat{S}(i)\big)}{\big[(-1)^{d+1}x\big]^{l}}.\label{MH2}
\end{gather}
\end{Theorem}

\begin{proof} We will use the hypergeometric representation~\eqref{new-h}. Consider the expression \linebreak $Q_n\big(x/m^{1/l}\big)$ which gives
\begin{gather*}
Q_{ml +i}\big(x/m^{1/l}\big) = \frac{x^i}{m^{i/l}}\, \pFq{1}{ld +l-1}{-m}{\big(\hat{S}(i)\big)}{\frac{\big[(-1)^{d+1}x\big]^l}{m}}.
 \end{gather*}
The latter formula can be rewritten in the form
 \begin{gather*}
 m^{i/l} Q_{ml +i}\big(x/m^{1/l}\big) = x^i\,\pFq{1}{ld +l-1}{-m}{\big(\hat{S}(i)\big)}{\frac{\big[(-1)^{d+1}x\big]^l}{m}}.
 \end{gather*}
Then formula \eqref{MH2} follows from \eqref{MH1}.
\end{proof}

It is worth noticing that the asymptotics depends on the remainder $n$ $({\rm mod}~l)$ which indicates that there is probably no general asymptotic formula, but it might exist for the subsequence with the same value of the remainder $n$ $({\rm mod}~l)$. This phenomenon is well-known in the case of Hermite polynomials, where the even-indexed and the odd-indexed polynomials have different asymptotics, see, e.g.,~\cite{AS}.

Notice that as we explained in Remark~\ref{GBess} the function
\begin{gather*}
	\pFq{0}{ld +l-1}{-}{\big(\hat{S}(i)\big)}{\big[(-1)^{d+1}x\big]^l}
\end{gather*}
	is also a generalized Bessel function in the sense of~\cite{BHY2}.
	
\section{Examples} \label{exa}

\begin{Example}[Gould--Hopper polynomials]\label{GH} Consider the simplest case $R(H) \in {\mathbb C}$, i.e., $G = \partial$. Then for $q = \tau G^l$, using equation~\eqref{Gl} we obtain the polynomial system
\begin{gather*}
		P_n(x) = x^n\,\pFq{l}{0}{\dfrac{-n}{l}, \ldots, \dfrac{-n+l-1}{l}}{-}{\tau\left(\frac{-l}{x}\right)^l}.
\end{gather*}
		These polynomials coincide with the well-known Gould--Hopper polynomials $g^l_n(x, \tau) $, cf.~\cite{GH,LCS}.
		
According to our scheme they are the eigenfunctions of the differential	operator
\begin{gather*}
		L = l\tau\partial^{l} + x\partial
\end{gather*}
		and satisfy the recurrence relation
\begin{gather*}
		xg^l_n(x,\tau)= g^l_{n+1}(x,\tau) - \tau l n(n-1)\cdots (n-l +2) g^l_{n-l+1}(x,\tau).
\end{gather*}
		
Using the second form of hypergeometric representation \eqref{new-h} we can also express them as
\begin{gather*}
g_{ml +i}^l(x) =x^i\pFq{1}{l-1}{-m }{\big(\hat{S}(i)\big)}{-\tau\left(\frac{x}{l}\right)^l}.
\end{gather*}
		
Observe that for $l=2$, these polynomials coincide (up to rescaling) with the classical Hermite polynomials. Eventually the Gould--Hopper polynomials turned out to be quite useful in quantum mechanics, integrable systems (Novikov--Vesselov equation), combinatorics, etc., see, e.g., \cite{Cha,DLMTC, VL}.
		
The cases with $G = R(H)\partial$ with arbitrary $R$ can be considered as generalizations of the Gould--Hopper polynomials. In this situation we use $q(G) = G^l$, where $G = R(H)\partial$, $R$ being a~polynomial of an arbitrary degree. The corresponding expression for these polynomials provided by~\eqref{Gl} is as follows
\begin{gather*}
P_n(x) = x^n \,\pFq{dl+l}{0}{\Delta(l; -n) , (\Delta(l; -n-\alpha))} {-}{\left(\frac{(-l)^{d+1}\rho}{x}\right)^l}.
\end{gather*}
		
We can explicitly write discrete analogs of the (generalized) Gould--Hopper polynomials. Namely,
\begin{gather*}
P_n(x) = (x)_n \,\pFq{dl+l}{l}{\Delta(l; -n), (\Delta(l; -n-\alpha))} {\Delta(l; x-n)}{ \big((-1)^{(d+1)}l^d\rho\big)^l}. 
\end{gather*}
		
The most straightforward analog, which one might call the {\it discrete Gould--Hopper polynomials}, corresponds to $G = \tau \Delta^l$ ($d = 0$). These polynomials have the following hypergeometric representation
\begin{gather*}
P_n(x) = (x)_n \,\pFq{l}{l}{\Delta(l; -n)} {\Delta(l; x-n)}{(-l\rho)^l}. 
\end{gather*}
Having in mind the existing applications of the Gould--Hopper polynomials it is worth checking whether their generalized versions have similar or other applications.
\end{Example}

\begin{Example}[Konhauser--Toscano polynomials]\label{Kon-Tos} In \cite{Kon} Konhauser has defined two families of polynomials denoted by $Y^{\alpha}_n(x;l)$ and $Z^{\alpha}_n(x;l)$, $n = 0, \ldots$, where $\alpha \in {\mathbb R}$, $l \in {\mathbb N}$. The polynomials~$Z^{\alpha}_n(x;l)$ are in fact polynomials in~$x^l$. The polynomials~$Y^{\alpha}_n(x;l)$ are polynomials in the original variable~$x$.
		
These two families are biorthogonal with respect to the weight function corresponding to the Laguerre polynomials, i.e.,
\begin{gather*}
		\int_{0}^{\infty}x^{\alpha}e^{-x}Y^{\alpha}_n(x;l) Z^{\alpha}_m(x;l){\rm d}x = h_m\delta_{n,m}
\end{gather*}
with $h_m \neq 0$. The polynomials $Z^{\alpha}_n(x;l)$ were introduced earlier by Toscano, see~\cite{Tos}. Their hypergeometric representation
\begin{gather*}
Z^{\alpha}_n(x;l) = \dbinom{\alpha + ln}{ln} \frac{(ln)!}{n!} \, \pFq{1}{l}{-n}{\Delta(l;\alpha+1)}{(x/l)^l},
\end{gather*}
found in \cite{LCS} shows that they can be constructed using the methods of the present paper. Let us consider $G= R(H)\partial$, withl $R(H) = \prod\limits^{l}_{s=1}\big(H+\frac{\alpha +s}{l}\big)$. Then
\begin{gather*}
		Z_n^{\alpha}(x; l) = P_n\big((x/l)^l\big).
\end{gather*}
In the case $l=2$ the polynomials were discovered much earlier by L.V.~Spencer and U.~Fano~\cite{SF} in their studies of the $X$-rays diffusion.
\end{Example}
		
\begin{Remark}\label{Rkon}
		In \cite{Kon} Konhauser has proven that the polynomials $Y^{\alpha}_n(x;l)$ are the eigenfunctions of a differential operator of order $l+1$ and that they satisfy a $(l+2)$-recurrence relation of the form
\begin{gather*}
		x^l Y^{\alpha}_n(x;l) = \sum\limits_{j=-1}^{l} \gamma_j(n) Y^{\alpha}_{n+j}(x;l).
		\end{gather*}
		The polynomials $Z^{\alpha}_n(x;l)$ satisfy a relation of the form
\begin{gather*}
		x^l Z^{\alpha}_n(x;l) = \sum\limits_{j=-1}^{l} \beta_j(n) Z^{\alpha}_{n-j}(x;l).
		\end{gather*}
(This relation follows from the definition $Z_n(x;l) := P^R_n\big((x/l)^l\big)$ and the properties of \linebreak $P^R_n\big((x/l)^l\big)$.)

As mentioned earlier the Konhauser--Toscano polynomials have applications to random mat\-rix theory. The so-called Borodin--Muttalib ensembles~\cite{Bor, Mu} are based on their biorthogonality.
\end{Remark}
		
In fact the polynomials $ Y^{\alpha}_n(x;l)$, $n=0, 1,\ldots, l-1$ are closely related with the functionals, which define the $d$-orthogonal polynomials $P_n(x)$.
		
It is quite interesting to define discrete analogs for the above Konhauser polynomials.

\begin{Example}[matching polynomials of graphs]\label{graph} Polynomial systems considered in this example are taken from \cite{AEMR} and they are relevant for the so-called chemical graph theory.
		
We need some notions from the graph theory, see, e.g., \cite{AEMR, Di, Farr} and the references therein. Let $K$ be a connected graph with $n$ vertices. Following~\cite{RMA} we define the higher Hosoya number~$p_r(K, j)$ as the number of ways one can select $j$ non-incident paths of length~$r$ in~$K$. Using the Hosoya numbers, the higher-order matching polynomial $M_r(K)$ of $K$ is defined by the relation
\begin{gather*}
		M_r(K) := \sum\limits_{j=0}^{ \lceil{\frac{n}{r+1}}\rceil } (-1)^jp_r(K, j)x^{n-(r+1)j},
\end{gather*}
see \cite{AEMR,Farr}. When $K = K_n$ is the complete graph on $n$ vertices (i.e., each pair of vertices is connected by an edge) the corresponding polynomials were explicitly computed in~\cite{AEMR, Farr}. Using combinatorial arguments it was shown that these polynomials are given by
\begin{gather*}
		M_{r} (K_n) = \sum\limits_{j=0}^{n} (-1)^j\frac{n!x^{n- (r+1)j}}{(n- (r+1)j)! j!2^j }.
\end{gather*}
If we compare this expression with
\begin{gather*}
P_n = e^{-G^{r+1}/2}x^n,
\end{gather*}		
where $G = \partial_x$, we see that they coincide. Hence we obtain a hypergeometric representation
\begin{gather*}
		M_r(K_n) = x^n \pFq{r+1}{0}{\Delta(r+1; -n)}{-}{ \frac{(-1)^r(r+1)^{r+1}}{2x^{r+1}}}.
\end{gather*}
(This representation was earlier found in \cite{AEMR}.) We see that the matching polynomials of complete graphs are the eigenfunctions of a linear differential operator and that they satisfy an $(r+2)$-term recurrence relation which can be useful in the studies of these polynomials.
		
The case $r = 1$ deserves a special attention since it corresponds to the Hermite polynomials as was observed long ago, e.g., in~\cite{AEMR, Gut}. The corresponding coefficients~$p_1(K,j)$ are the original Hosoya numbers~\cite{Hos}.
		
Let us consider another example of matching polynomials, this time of complete bipartite graphs $K_{n,m}$, $n \geq m$ with $n+m$ vertices. (Recall that~$K_{n,m}$ is a graph whose vertices are split into two nonintersecting sets $V_n$ and $V_m$ with $n$ and $m$ elements resp. and every vertex in $V_n$ is connected to every vertex in~$V_m$. We consider the case when~$r$ is odd. Then~\cite{AEMR, Farr} contain the formula
\begin{gather*}
				M_r(K_{n, m}) = \sum\limits_{j=0}^{n} \binom{n}{jr} \binom{m}{jr}\frac{(-1)^j((jr)!)^2}{j!} x^{n- 2rj},
\end{gather*}
which can easily be transformed into
\begin{gather*}
M_r(K_{n,m}) = x^{n+m}\, \pFq{r+1}{0}{\Delta(r+1; -n), \Delta(r+1; -m) }{-}{ \frac{-(r+1)^{r+1}}{(2x)^{r+1}}}.
\end{gather*}
Set $m = n - M$, $M\geq 0$. Using this notation, we see that
\begin{gather*}
		M_r(K_{n,m}) = x^{2n - M}\, \pFq{r+1}{0}{\Delta(r+1; -n), \Delta(r+1; - n +M ) }{-}{ \frac{-(r+1)^{r+1}}{(2x)^{r+1}}}.
		\end{gather*}
In other words, $M_r(K_{n,m}) $ coincide with $x^{n-M}P_n(x)$, where the polynomials $P_n(x)$ are constructed via	$q(G) = -G^{r+1}/2$ and $G = (x\partial - M)\partial$.
			
While the hypergeometric representation is known, the properties of these polynomials listed in Theorem~\ref{any} are new. In particular, the differential equation and the recurrence relations they satisfy seem to be new.
		
Notice that for $r=1$ (i.e., when all the paths are edges), these polynomials coincide with $L^{(M)}_n\big(x^2\big)$, where $L^{(\alpha)}_n(y)$ are the generalized Laguerre polynomials.
\end{Example}
		
The above result suggests a conjecture about the matching polynomials of complete $k$-partite graphs, i.e., graphs whose vertices can be colored into~$k$ distinct colors, so that the two endpoints of every edge have different colors. By a complete $k$-partite graphs we mean that any two vertices with different colors are connected by an edge, see more details in~\cite{CZ}.
		
Namely, consider all graphs with $N=n_1 + \cdots + n_k$ vertices. Denote the corresponding $k$-partite graph by $K_{(n)}$.
		
\begin{Conjecture}\label{n-part} For odd $r$, the matching polynomials $M_r(K_{(n)})$ are given by
\begin{gather*}
		M_r(K_{(n)}) = x^N \,\pFq{r+1}{0}{\Delta(r+1; -n_1)\ldots \Delta(r+1; -n_k) }{-}{ \frac{-(r+1)^{r+1}}{(2x)^{r+1}}}.
\end{gather*}
\end{Conjecture}		
		
More examples can be found in the cited papers.
		
\subsection*{Acknowledgements}
	
The author is sincerely grateful to Boris Shapiro for sharing and discussing some polynomial systems studied here. Without this the current project would probably have never seen the light of the day. Also his advises for improvement of the text are acknowledged. The author wants to thank the Mathematics Department of Stockholm University for the hospitality in April 2015 and April 2017. Last but not least the author acknowledges extremely helpful suggestions and corrections made by the referees which helped to improve considerably the text. This research has been partially supported by the Grant No DN 02-5 of the Bulgarian Fund ``Scientific research''.

\pdfbookmark[1]{References}{ref}
\LastPageEnding

\end{document}